\documentclass[a4paper, 11pt]{article}
\usepackage[psamsfonts]{amsfonts}
\usepackage{amsmath,amssymb,amsthm,dsfont}
\usepackage{graphicx}  
\usepackage{float}
\usepackage{mathrsfs}
\usepackage{subcaption}
\usepackage{hyperref}
\usepackage{xcolor}

\hypersetup{
    colorlinks=true,
    linkcolor=blue,
    urlcolor=red,
    citecolor=red
}
\newcommand{\kP}{\mathcal{P}}

\newcommand{\kS}{\mathcal{S}}

\newtheorem{theorem}{Theorem}[section]

\newtheorem{lemma}[theorem]{Lemma}
\newtheorem{corollary}[theorem]{Corollary}

\newtheorem{proposition}[theorem]{Proposition}

\newcommand{\be}{\begin{equation}}
\newcommand{\ee}{\end{equation}}

\newcommand{\arctanh}{{\rm atanh}}

\numberwithin{equation}{section}

\newcommand{\rem}[1]{}

\def\1{{\mathchoice {1\mskip-4mu\mathrm l}      
{1\mskip-4mu\mathrm l}
{1\mskip-4.5mu\mathrm l} {1\mskip-5mu\mathrm l}}}

\newcommand{\col}[1]{\textcolor[rgb]{0,0,0}{#1}}

\newcommand{\eqan}[1]{\begin{align}#1\end{align}}

\DeclareMathSymbol{\dagger}{\mathbin}{letters}{"8E}

\def\be#1{\begin{equation*}#1\end{equation*}}

\title{Phase diagram and fluctuations in the cubic mean field Ising model}
\author{Van Hao Can\thanks{Institute of Mathematics, Vietnam Academy of Science and Technology. Email: cvhao@math.ac.vn}
\and Viet Tinh Dang\thanks{University of Khanh Hoa and Institute of Mathematics, Vietnam Academy of Science and Technology. Email: dangviettinh@ukh.edu.vn}}
\date{}
\begin{document}
\maketitle

\begin{abstract}
We study the cubic mean field Ising model with three body and two body interactions and an external field. Our main result is a complete classification of the global maximizers of the variational pressure in the full three-parameter region $(K,J,h)\in(0,\infty)\times\mathbb R^2$. We identify a coexistence surface with two nondegenerate global maximizers and a critical curve with a unique quartically degenerate maximizer; elsewhere the maximizer is unique and nondegenerate. This extends the phase diagrams previously obtained on two-parameter subfamilies \cite{CMO24,O24}. As applications of the general results in \cite{CR}, we give explicit magnetization fluctuation theorems and asymptotic distributions for one-parameter maximum likelihood estimators in the resulting regimes.
\end{abstract}

\section{Introduction}\label{sec:introduction}
The Ising model was introduced to describe cooperative behavior and phase transitions in ferromagnetic systems and has since become a central model of interacting spins on graphs; see, for example, \cite{E85,N05,N09}. Its classical Hamiltonian contains pairwise interactions, whereas higher-order Ising models allow simultaneous interactions among three or more spins. Such interactions have been considered in statistical physics and in models of complex networks, where pairwise dependencies may not capture all relevant features; see \cite{B20,HBH89,SL99,S72,Y19,CKO22}. In particular, cubic interactions break spin flip symmetry and may change the nature and location of phase transitions. On general graphs, the analysis of multi-spin interactions is delicate, and the correlation inequalities commonly used for ferromagnetic pair interaction models do not directly provide a general theory. The mean-field setting offers an analytically tractable starting point, while retaining nontrivial competition between the interaction and entropy terms.

The cubic mean-field Ising model has already been
investigated on important two-parameter subfamilies.
Contucci, Mingione and Osabutey \cite{CMO24} studied
the interaction parameters $(K,J)$ in zero external
field ($h=0$), establishing the phase diagram and
limit theorems for the magnetization.
For the same model, Eichelsbacher \cite{E24}
applied Stein's method to obtain quantitative
bounds for the magnetization fluctuations,
together with concentration and moderate
deviation results. The model with a cubic interaction and an external field, but without the quadratic interaction ($J=0$), is a special case of the $p$-spin Curie--Weiss model studied by Mukherjee, Son and Bhattacharya \cite{MSB21,MSB25}. Its phase properties and critical exponents were also investigated by Osabutey \cite{O24}. 
Contucci, Osabutey and Vernia \cite{COV23} studied the inverse problem for the full three-parameter cubic model. They derived formulas for reconstructing the model parameters from the magnetization, susceptibility and third cumulant, and investigated the performance of their method numerically.

These studies motivate the \emph{simultaneous}
variation of all three parameters. It is not
enough simply to combine two previously studied
planar diagrams: one must determine which
stationary branches are globally maximizing in
three-dimensional parameter space, locate their
equal-height locus, and describe how this locus
terminates at critical degeneracy. Our principal result provides this global geometric classification for  $(K,J,h)\in(0,\infty)\times\mathbb R^2$. The scaling limit for magnetization and the maximum likelihood estimators for parameters   are then applications of the general mean field results of Can and R\"ollin \cite{CR}. 

\paragraph{The model and free energy.}
For $N\ge 1$, let $\sigma=(\sigma_1,\ldots,\sigma_N)\in\{-1,+1\}^N$, and denote the magnetization by
\begin{equation}\label{def:mN-intro}
m_N(\sigma)=\frac{1}{N}\sum_{i=1}^N\sigma_i.
\end{equation}
We consider the Hamiltonian
\begin{equation}\label{Hamiltonian}
\begin{split}
H_N(\sigma)
&=\frac{K}{3N^2}\sum_{i,j,k=1}^N\sigma_i\sigma_j\sigma_k
+\frac{J}{2N}\sum_{i,j=1}^N\sigma_i\sigma_j
+h\sum_{i=1}^N\sigma_i\\
&=N\left(\frac K3m_N^3+\frac J2m_N^2+hm_N\right).
\end{split}
\end{equation}
The inverse temperature has been absorbed into the parameters $K,J,h$. The Gibbs measure and the partition function are given by
\begin{equation}\label{BG measure}
\mu_{N,K,J,h}(\sigma)=\frac{1}{Z_N(K,J,h)}e^{H_N(\sigma)},
\qquad Z_N(K,J,h)=\sum_{\sigma\in\{-1,+1\}^N}e^{H_N(\sigma)}.
\end{equation}
Simultaneous reversal of all spins transforms $(K,J,h)$ into $(-K,J,-h)$. Consequently, it suffices to analyze $K>0$ to obtain the corresponding $K<0$ results by symmetry. Throughout, our phase diagram is stated on the open region
\begin{equation}\label{parameter-region}
\kP=(0,\infty)\times\mathbb R\times\mathbb R.
\end{equation}
The free energy is then defined by  $$p_N(K,J,h)=\frac{1}{N}\log Z_N(K,J,h).$$
By the standard variational principle (see \cite{E85,DZ98}), its thermodynamic limit is
\begin{equation}\label{theormodynamic limit}
\begin{split}
p(K,J,h)&:=\lim_{N\to\infty}p_N(K,J,h)
=\max_{m\in[-1,1]}\phi_{K,J,h}(m),\\
\phi_{K,J,h}(m)&=\frac K3m^3+\frac J2m^2+hm-I(m),
\end{split}
\end{equation}
where the binary entropy contribution is
\begin{equation}\label{entropy}
I(m)=\frac{1-m}{2}\log\frac{1-m}{2}
+\frac{1+m}{2}\log\frac{1+m}{2},\qquad m\in[-1,1],
\end{equation}
with the convention $0\log0=0$. Thus, the equilibrium magnetizations are the global maximizers of the one dimensional function $\phi_{K,J,h}$. Moreover, it turns out that their number and degeneracy determine both the macroscopic phase structure and the scale of model's fluctuations.

This variational description has a direct probabilistic interpretation. The number of spin configurations with magnetization $m=2k/N-1$ is $\binom{N}{k}$, and its exponential growth contributes the entropy term $-I(m)$. Therefore, the law of $m_N$ is governed, on the exponential scale, by the competition between energy and entropy encoded by $\phi_{K,J,h}$. A unique global maximizer corresponds to a single thermodynamic phase, whereas several maximizers of equal height produce macroscopic phase coexistence. Locally, the first nonzero term in the Taylor expansion at a maximizer determines the fluctuation scale. These facts also explain why determining the \emph{global} rather than merely local extrema is essential for the subsequent probabilistic results.

A global maximizer $m$ is called \emph{$2k$-regular} if
$\phi_{K,J,h}^{(j)}(m)=0$ for $1\le j<2k$ and
$\phi_{K,J,h}^{(2k)}(m)<0$.
As shown in Section 2, every global maximizer in the present model is either \(2\)-regular or \(4\)-regular. The main problem is therefore to determine the number and regularity of the global maximizers throughout the three dimensional parameter space.

\paragraph{Main result: the full phase diagram.}
Our main result provides a complete classification of the global
maximizers of $\phi_{K,J,h}$ in the three dimensional parameter
space $\mathcal P$. The classification is governed by the
concavity properties of $\phi_{K,J,h}$. Indeed, its second
derivative is given by
\[
\phi''_{K,J,h}(m)
=J-\left(\frac{1}{1-m^2}-2Km\right).
\]
This motivates the definition
\begin{equation}\label{tau-intro}
\tau(K):=\min_{-1<m<1}
\left\{\frac{1}{1-m^2}-2Km\right\},
\qquad K>0.
\end{equation}
As shown in Section~\ref{sec:stationary}, the minimum is
attained at a unique point in $(-1,1)$.

For $J\leq\tau(K)$, the function $\phi_{K,J,h}$
has a unique global maximizer for every $h\in\mathbb R$.
When $J>\tau(K)$, the function may have two distinct local
maxima, depending on the external field $h$.
As $h$ varies, the two maxima attain the same height
at a unique value $h=S(K,J)$. The corresponding parameters
form the coexistence surface $\mathcal S$.

A critical curve $\gamma$ arises at the boundary of this
surface, where the two competing maxima merge. It is
characterized by the simultaneous vanishing of the first
three derivatives of $\phi_{K,J,h}$ and consists of parameters
for which the unique global maximizer is $4$-regular.
The following theorem gives a precise description of these
sets and the global maximizers in all parameter regimes.

\begin{theorem}[Full three-parameter phase diagram]
\label{Phase diagram}
There exists a unique smooth function
\[
S:\bigl\{(K,J):K>0,\ J>\tau(K)\bigr\}
\rightarrow\mathbb R
\]
such that $h=S(K,J)$ is exactly the field at which
$\phi_{K,J,h}$ has two distinct global maximizers.
Define the coexistence surface by
\begin{equation}\label{defsur}
\mathcal S
=\{(K,J,S(K,J)):K>0,\ J>\tau(K)\}
\end{equation}
and the critical curve by
\begin{equation}\label{defgamma}
\gamma=
\left\{
\left(
\frac{t}{(1-t^2)^2},
\frac{1-3t^2}{(1-t^2)^2},
\operatorname{atanh}(t)
-\frac{t-2t^3}{(1-t^2)^2}
\right):0<t<1
\right\}.
\end{equation}
Then the global maximizers of $\phi_{K,J,h}$
are completely classified as follows.
\begin{enumerate}
\item[(a)] \emph{Coexistence.}
If $(K,J,h)\in\mathcal S$, there are exactly two
distinct global maximizers $m_1<m_2$, both $2$-regular.

\item[(b)] \emph{Criticality.}
If $(K,J,h)\in\gamma$, there is a unique global
maximizer $m_*$, which is $4$-regular.

\item[(c)] \emph{Regular single phase.}
If $(K,J,h)\in\mathcal P\setminus
(\mathcal S\cup\gamma)$, there is a unique global
maximizer $m_*$, which is $2$-regular.
\end{enumerate}
Moreover, the closure of the coexistence surface
within $\mathcal P$ satisfies
\[
\overline{\mathcal S}^{\,\mathcal P}
=\mathcal S\cup\gamma.
\]
\end{theorem}

\begin{figure}[tbp]
\centering
\includegraphics[width=.62\linewidth]{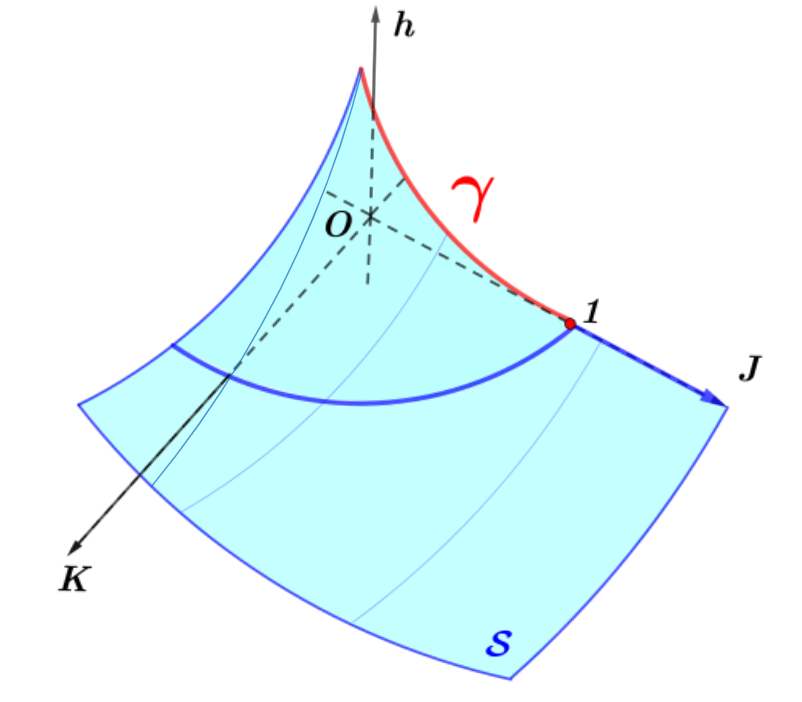}
\caption{Schematic phase diagram in the $(K,J,h)$ parameter space. The surface $\mathcal S$ represents coexistence of two nondegenerate global maximizers, while the curve $\gamma$ represents a unique quartically degenerate maximizer. }
\label{phase diagram}
\end{figure}

\begin{figure}[tbp] 
    \centering
    
    \begin{subfigure}{0.4\textwidth}
        \centering
        \includegraphics[width=\linewidth]{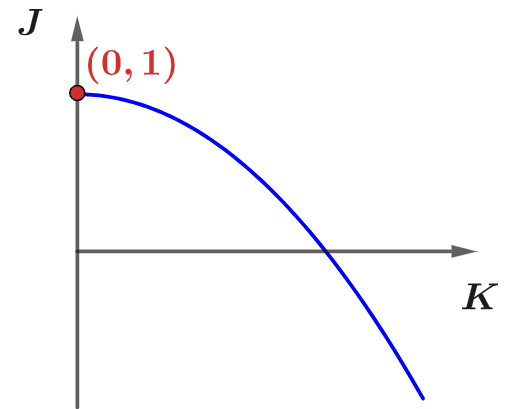}
        \caption{}
        \label{}
    \end{subfigure}
    \hspace{1.75cm}
    \begin{subfigure}{0.38\textwidth}
        \centering
        \includegraphics[width=\linewidth]{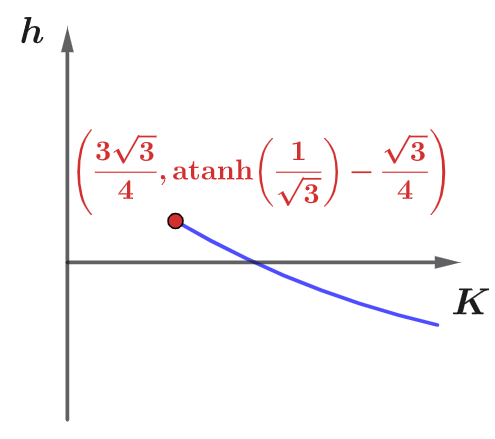}
        \caption{}
        \label{}
    \end{subfigure}
    \caption{The intersections of the surface $\mathcal S$ with (a) the plane $h=0$ and (b) the plane $J=0$. In each case, the blue curve  represents the intersection while the red dot is the critical point.} 
    \label{Figure2}
\end{figure}

The coexistence surface $\mathcal S$ describes a first order phase transition. At each point of $\mathcal S$, the two global maximizers $m_1<m_2$ correspond to distinct equilibrium phases. In particular, the limiting pressure is not differentiable with respect to $h$ across this
surface. The critical curve $\gamma$ forms the boundary of $\mathcal S$, where the two phases merge into a single degenerate phase. Thus, the theorem extends the previously known two parameter phase diagrams to the full $(K,J,h)$-model.


\paragraph{Fluctuation of the magnetization.}
The asymptotic behavior of magnetization in mean-field models has been studied extensively, starting from the classical works of Ellis and Newman \cite{EN78} and Ellis, Newman and Rosen \cite{ENR80}; see also \cite{E85,MSB21} for further developments. More recently, Can and R\"ollin \cite{CR} provided a general framework with quantitative convergence bounds, which we apply to the present model. It turns out that fluctuation regimes are determined by  the maximizers and their regularity orders. In the regular single phase region, the quadratic expansion around the unique maximizer gives Gaussian fluctuations: $m_N$ concentrates exponentially around $m_*$ and
\begin{equation}\label{intro-mag-normal}
\sqrt N(m_N-m_*)\ \Rightarrow\
 N\!\left(0,-\frac{1}{\phi''_{K,J,h}(m_*)}\right).
\end{equation}
On the coexistence surface, there are two contributions of equal exponential order to the partition function. The magnetization therefore converges to a nontrivial mixture of point masses at $m_1$ and $m_2$; conditional on being near either maximizer, the quadratic expansion still gives Gaussian fluctuations on the $N^{-1/2}$ scale, with explicit variances and weights. On the critical curve the quadratic and cubic terms vanish at the unique maximizer. The negative quartic term becomes the first nonzero contribution, so the fluctuations are larger and
\begin{equation}\label{intro-mag-quartic}
N^{1/4}(m_N-m_*)\Rightarrow Y_*,
\qquad \mathbb P(Y_*\in dx)=C_*\exp\!\left\{
\frac{\phi^{(4)}_{K,J,h}(m_*)}{24}x^4\right\}\,dx.
\end{equation}
The general results also yield quantitative Wasserstein convergence rates: $O(N^{-1/2})$ for the Gaussian fluctuations and $O(N^{-1/4})$ on the critical curve. The complete statement, including exponential concentration, coexistence weights and conditional limits, is given in Theorem~\ref{limit:magn}.

\paragraph{Fluctuation of the estimators.} 
The asymptotic behavior of maximum likelihood
estimators in mean-field models has been investigated by Comets and Gidas \cite{CG91} for the Curie--Weiss model and, more recently, by Mukherjee, Son and Bhattacharya \cite{MSB25} for higher-order $p$-spin models. We apply the general results in \cite[Theorem~4.1]{CR}  to obtain the scaling limits of the estimators in the present model.
More precisely, we estimate each of the three model
parameters separately, treating the other two as know. Each parameter enters the Hamiltonian linearly, and  given a single spin configuration, the associated sufficient statistic is $m_N^3/3$, $m_N^2/2$, or $m_N$, according as one estimates $K$, $J$, or $h$. The MLE equates this observed statistic to its expectation under the fitted Gibbs measure. Its rate is determined by how sensitively that expectation responds to a local perturbation of the parameter. The general theory of \cite{CR} applies once the phase diagram and the nondegeneracy conditions at the relevant maxima have been checked.

Away from the exceptional parameter sets, the MLEs exhibit three distinct behaviors. In the regular single phase regime they have a Gaussian limit on the usual $N^{-1/2}$ scale. On the coexistence surface a positive or negative local change of parameter favors the well having, respectively, the larger or smaller sufficient statistic; this selection creates a non-Gaussian limit with an atom at zero and two one-sided Gaussian components. On the critical curve a parameter perturbation of size $N^{-3/4}$ changes the quartic limiting law by an order-one linear tilt, and hence gives the correct estimation scale $N^{-3/4}$, rather than the $N^{-1/4}$ magnetization scale. Precise distributions, hypotheses and exceptional sets are stated in Theorem~\ref{limit:mle}. 

\paragraph{Proof strategy and organization.}
Section~\ref{sec:phase} proves Theorem~\ref{Phase diagram}. The proof first analyzes the stationary equation and shows that its geometry is controlled by a strictly convex auxiliary function. It then compares the two local maximizing branches to construct the unique coexistence surface and identifies the critical curve through the vanishing of the first three derivatives of the variational pressure. Section~\ref{sec:applications} contains the full statements and proofs of the fluctuation for the magnetization and one-parameter MLE  using general results in \cite{CR}. 

\paragraph{Notation.}
We write $X_N\Rightarrow X$ for convergence in distribution
and $\mathcal L(X)$ for the law of $X$.
Let $N(\mu,\sigma^2)$ denote
the Gaussian distribution with mean $\mu$ and variance $\sigma^2$.
For probability laws with finite first moments, the
$1$-Wasserstein distance is defined by
\[
d_W(\mathcal L(X),\mathcal L(Y))
=\sup_{\operatorname{Lip}(f)\le 1}
\bigl|\mathbb E[f(X)]-\mathbb E[f(Y)]\bigr|.
\]

\section{The phase diagram}\label{sec:phase}
\subsection{Stationary points and regularity}\label{sec:stationary}
We begin by computing the derivatives of $\phi_{K,J,h}$ on $(-1,1)$:
\begin{align}\label{first der}
		& \phi_{K,J,h}'(m)=Km^2+Jm+h-\arctanh(m)\\ \label{second der}
        & \phi_{K,J,h}''(m)=2Km+J-\dfrac{1}{1-m^2}\\ 
		& \phi_{K,J,h}^{(3)}(m)=2K-\dfrac{2m}{(1-m^2)^2}\\
		& \phi_{K,J,h}^{(4)}(m)=-\dfrac{2(3m^2+1)}{(1-m^2)^3}<0,\;\;\; \forall m\in (-1,1). \label{fourth der}
	\end{align}
Every interior maximizer  satisfies the stationarity equation \begin{equation}\label{m=tanh}
	 \phi_{K,J,h}'(m) =	 Km^2+Jm+h-\arctanh(m)=0.
	\end{equation}   
The continuous function $\phi_{K,J,h}$ attains a maximum on $[-1,1]$.
Since $\phi'_{K,J,h}(m)\to+\infty$ as $m\downarrow-1$ and
$\phi'_{K,J,h}(m)\to-\infty$ as $m\uparrow1$, every maximizer lies
in $(-1,1)$ and satisfies \eqref{m=tanh}. Define
\begin{equation}
    g_{K,J}(m)=\arctanh(m)-Km^2-Jm, \;\; m\in(-1,1).
\end{equation}
Note that $\phi_{K,J,h}'(m)=h-g_{K,J}(m)$. Hence,  the maximizers of $\phi_{K,J,h}(m)$ can be characterized by  the intersection between the line $h$ and the curve $g_{K,J}(m)$. The following result gives some properties of the function $g_{K,J}(m)$.
   \begin{proposition}[Properties of $g_{K,J}(m)$]\label{Properties of g} For any $K>0$, define the function
   \begin{equation}
       F_K(m):=\dfrac{1}{1-m^2}-2Km,\;\; m\in (-1,1).
   \end{equation} 
   Then $F_K(m)$ attains its minimum at a unique point $\overline{m}_0=\overline{m}_0(K)\in (-1,1)$. Furthermore, writing
   \begin{equation}
       \tau(K)=\displaystyle\min_{m\in (-1,1)} F_K(m)=F_K(\overline{m}_0),
   \end{equation}
   we have  
   \begin{enumerate}
       \item [(a)] if $J<\tau(K)$, $g_{K,J}(m)$  strictly increases on $(-1,1)$;
       \item [(b)] if $J>\tau(K)$, $g_{K,J}(m)$  has a local minimum   $\underline{m}$ and a local maximum $\overline{m}$ satisfying $\overline{m}<\underline{m}$.
       \item[(c)] if $J=\tau(K)$, $g_{K,J}(m)$ increases on $(-1,1)$ and has a unique inflection point, which is $\overline{m}_0.$
   \end{enumerate}
    \end{proposition}
\begin{figure}[tbp] 
    \centering
    
    \begin{subfigure}{0.3\textwidth}
        \centering
        \includegraphics[width=\linewidth]{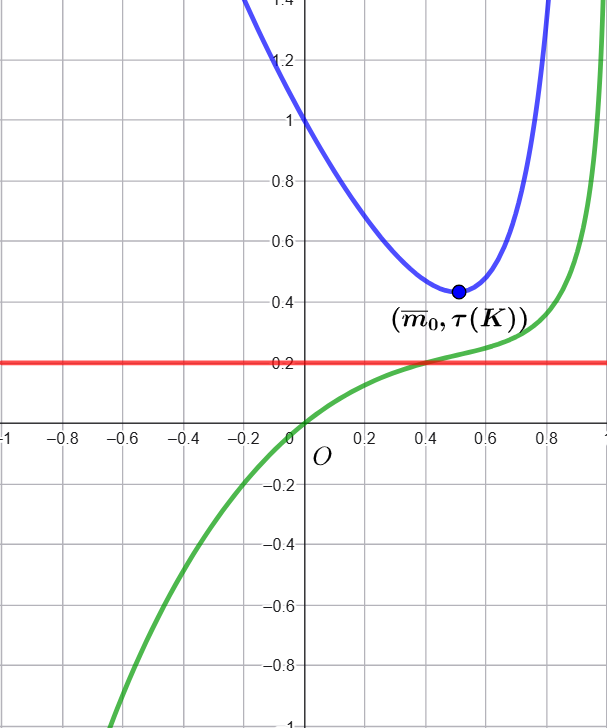}
        \caption{$J<\tau(K)$}
        \label{1a}
    \end{subfigure}
    \hfill
    \begin{subfigure}{0.25\textwidth}
        \centering
        \includegraphics[width=\linewidth]{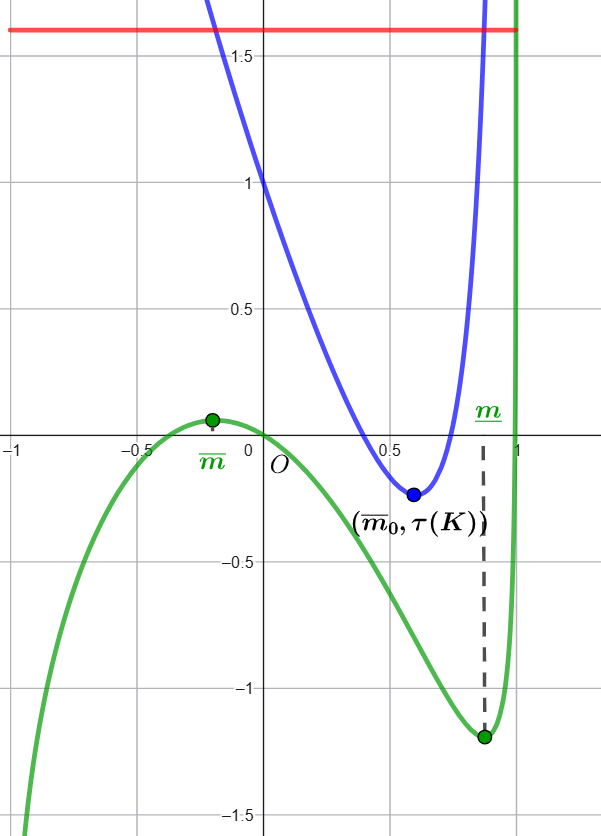}
        \caption{$J>\tau(K)$}
        \label{1b}
    \end{subfigure}
    \hfill
    \begin{subfigure}{0.3\textwidth}
        \centering
        \includegraphics[width=\linewidth]{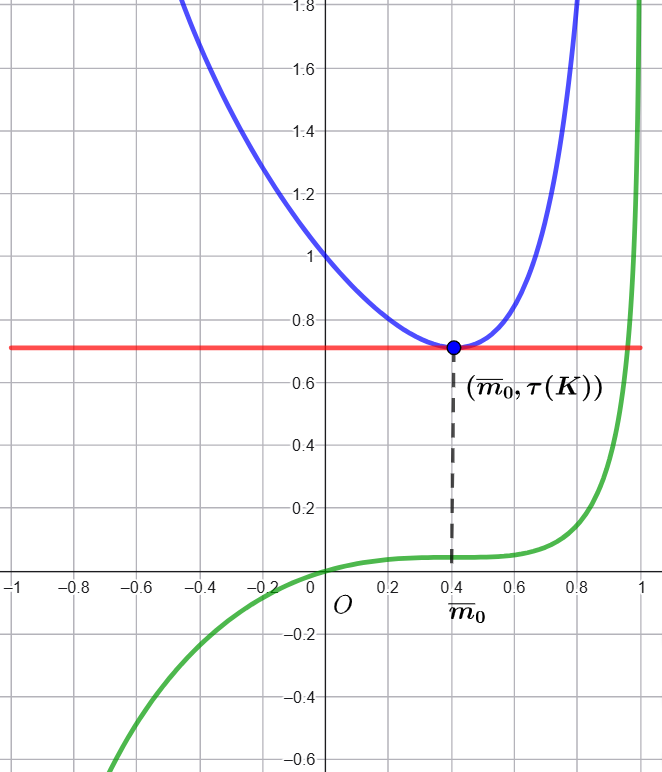}
       \caption{$J=\tau(K)$}
        \label{1c}
    \end{subfigure}

    \caption{Graph of functions $g_{K,J}(m)$ (the green curve), $F_K(m)$ (the blue curve) and $J$ (the red line) in cases: (a) $J<\tau(K)$, (b) $J>\tau(K)$ and (c) $J=\tau(K)$.} 
    \label{Figure1}
\end{figure}
    \begin{proof} First, we prove that $F_K(m)$ has the unique  minimizer. Indeed, notice that by  \eqref{second der}
    \begin{equation}
       \phi_{K,J,h}''(m)= J-F_K(m),
    \end{equation}
and hence, for every $m\in(-1,1)$,
    \begin{align}
    F'_K(m)&=-\phi_{K,J,h}^{(3)}(m)\\
  \label{f''}  F''_K(m)&= -\phi_{K,J,h}^{(4)}(m)>0.
    \end{align}
    Thus $F'_K(m)$  strictly increases on $(-1,1)$ and
    \begin{equation}
        \lim_{m\rightarrow -1^+}F'_K(m)= -\infty \text{ and } \lim_{m\rightarrow 1^-}F'_K(m)=+\infty.
    \end{equation} Therefore, the function $F'_K(m)$ changes its sign from negative to positive at a unique point, denoted by $\overline{m}_0$, which is the unique minimizer of $F_K(m)$ on $(-1,1)$.\\
    Next, since 
    \begin{equation}\label{g'}
        g'_{K,J}(m)= \dfrac{1}{1-m^2}-2Km-J= F_K(m)-J,
    \end{equation}
    the critical points of $g_{K,J}(m)$ are the intersection points between the line $J$ and the curve $F_K(m)$.
   \begin{enumerate}
       \item [(a)] If $J<\tau(K)$ (see Figure \ref{1a}), $g_{K,J}(m)$  strictly increases on $(-1,1)$ since  $F_K(m)-J\ge \tau(K)-J>0$ for all $m\in(-1,1)$.
       \item[(b)] If $J>\tau(K)$ (see Figure \ref{1b}), we can check that $F_K(m)$  strictly decreases on $(-1, \overline{m}_0]$,  strictly increases on $[\overline{m}_0,1)$ and $\lim_{m\rightarrow -1^+}F_K(m)= \lim_{m\rightarrow 1^-}F_K(m)=+\infty$. Therefore, the line $J$ intersects $F_K(m)$ at two distinct points $\overline{m}$ and $ \underline{m}$, which satisfy
       \begin{equation}
           \overline{m}<\overline{m}_0< \underline{m}.
       \end{equation}
       On the other hand, the line $J$ lies above the curve $F_K(m)$ on  $(\overline{m},\underline{m})$ and lies below this curve otherwise. Thus, using \eqref{g'}
       \begin{align}
         g'_{K,J}(m)   \begin{cases}
   <0, \text{if } m\in (\overline{m},\underline{m})\\
   =0, \text{if } m\in \{\overline{m},\underline{m}\}\\
   >0,  \text{otherwise}.
\end{cases}
       \end{align}
Therefore, $\overline{m}$ and $\underline{m}$ are the local maximizer and minimizer of $g_{K,J}(m)$, respectively.
  \item[(c)] If $J=\tau(K)=F_K(\overline{m}_0)$ then $g'_{K,J}(m)=0$ has a unique solution, which is $\overline{m}_0$. Recall that $F'_K(m)$ changes its sign from negative to positive at $\overline{m}_0$. Combining this fact  with $g''_{K,J}(m)=F'_K(m)$, we can conclude that $\overline{m}_0$ is the inflection point of $g_{K,J}(m).$
   \end{enumerate}
    \end{proof}
    
      \begin{corollary}\label{Regularity of maximizers} Consider the function  $\phi_{K,J,h}(m)$ with $(K,J,h)\in \mathcal{P}$.
        \begin{enumerate}
            \item [(i)] If $J<\tau(K)$  then $\phi_{K,J,h}''(m)<0$ for all $m\in (-1,1)$.
            \item[(ii)] If $J>\tau(K)$ then
            \begin{equation}\label{sign of phi''}
               \phi''_{K,J,h}(m) \begin{cases}
                    =0, \text{ if } m\in \{\overline{m},\underline{m}\}\\
                    >0, \text{ if } m\in (\overline{m},\underline{m})\\
                    <0, \text{ otherwise}.
                \end{cases}
            \end{equation}
            \item[(iii)] If $J=\tau(K)$ then $\phi_{K,J,h}''(\overline{m}_0)=0$ and $\phi_{K,J,h}''(m)<0$ when $m\ne \overline{m}_0$.
        \end{enumerate}
    \end{corollary}
    \begin{proposition}[Classification of critical points]\label{Classification of critical points} For any $(K,J,h)\in\mathcal{P}$, let $\tau(K)$, $\overline{m},\underline{m}$ and $\overline{m}_0$ be identified as in Proposition \ref{Properties of g}. Then the solutions of \eqref{m=tanh} are described as follows.
        \begin{enumerate}
            \item [(a)] Case $J<\tau(K)$. There exists a unique solution $m_0$, which is the global maximizer of $\phi_{K,J,h}(m)$.
             \item [(b)] Case $J>\tau(K)$. Let $\overline{\Psi}(K,J)=g_{K,J}(\overline{m})$ and $\underline{\Psi} (K,J)=g_{K,J}(\underline{m})$. Then 
             \begin{enumerate}
                 \item [(b1)] if $h<\underline{\Psi}(K,J)$ (resp. $h>\overline{\Psi}(K,J)$), there exists a unique solution $m_1$ (resp. $m_2$), which is the  global maximizer of $\phi_{K,J,h}(m)$;
                   \item [(b2)] if $\underline{\Psi}(K,J)<h<\overline{\Psi}(K,J)$, there exist three solutions  $m_1,m_2$ and $m_3$ (assume that $m_1<m_3<m_2)$, which satisfy:  $m_1$ and $m_2$ are local maximizers, while $m_3$ is a local minimizer  of $\phi_{K,J,h}(m)$;
                   \item [(b3)] if $h=\overline{\Psi}(K,J)$ (resp. $h=\underline{\Psi}(K,J)$), there exist  two solutions $m_2>\overline{m}$  (resp. $m_1<\underline{m}$), where $m_2$ (resp. $m_1$)  is  the unique global maximizer, while $\overline{m}$ (resp. $\underline{m}$) is the inflection point of $\phi_{K,J,h}(m)$.
             \end{enumerate}
             \item[(c)] Case $J=\tau(K)$. 
            \begin{enumerate}
                \item [(c1)] If $h= g_{K,J}(\overline{m}_0)$, there exists a unique solution $\overline{m}_0$, which is a global maximizer of $\phi_{K,J,h}(m)$.
                \item[(c2)] If $h\ne g_{K,J}(\overline{m}_0)$,  there exists a unique solution $m_0\ne \overline{m}_0$, which is a global maximizer of $\phi_{K,J,h}(m)$.
            \end{enumerate}
        \end{enumerate}
    Furthermore, for $i\in\{0,1,2\}$, $m_i$ is 2-regular, while $\overline{m}_0$ is 4-regular.
    \end{proposition}
     \begin{proof}
     We first recall that  the function $\phi_{K,J,h}(m)$ is continuous on the compact set $[-1,1]$ and cannot attain its maximum at $1$ or $-1$. Hence,  if $\phi_{K,J,h}(m)$ has exactly one critical point, then this point is the unique global maximizer of $\phi_{K,J,h}(m)$.
  \begin{enumerate}
      \item [(a)] Case $J<\tau(K)$. The function $g_{K,J}(m)$ strictly increases  on $(-1,1)$ by  Proposition \ref{Properties of g}. Therefore, for any $h\in\mathbb{R}$, \eqref{m=tanh} has a unique solution $m_0$, which is the unique global maximizer of $\phi_{K,J,h}(m)$. 
      \item [(b)] Case $J>\tau(K)$.
    \begin{enumerate}
        \item [(b1)] If $h<\underline{\Psi}(K,J)$ then \eqref{m=tanh} has  a unique solution $m_1<\overline{m}$, which is the unique global maximizer of $\phi_{K,J,h}(m)$. Similarly, if $h>\overline{\Psi}(K,J)$, then $\phi(m)$ has a unique global maximizer $m_2>\underline{m}$.
        \item [(b2)] If $\underline{\Psi}(K,J)<h<\overline{\Psi}(K,J)$ then \eqref{m=tanh} has three solutions $m_1,m_2$ and $m_3$ satisfying
        \begin{equation}
            m_1<\overline{m}<m_3<\underline{m}<m_2.
        \end{equation}
       By Corollary \ref{Regularity of maximizers}, $m_1$ and $m_2$ are local maximizers, while $m_3$ is the local minimizer of $\phi_{K,J,h}(m)$.
        \item[(b3)] If $h=\overline{\Psi}(K,J)$ then \eqref{m=tanh} has two solutions $\overline{m}<m_2$. Since $\overline{m}$ is the solution of $J=F_K(m)$, we have $\phi''_{K,J,h}(\overline{m})=0.$ Furthermore, since
        \begin{align}
            \phi_{K,J,h}^{(3)}(\overline{m})&=-F'_K(\overline{m})>0\\ \label{eq5}
            \phi''_{K,J,h}(m_2)&=J-F_K(m_2)<0,
        \end{align}
        $\overline{m}$ is the inflection point, while $m_2$ is the unique global maximizer  of $\phi_{K,J,h}(m)$. The case $h=\underline{\Psi}(K,J)$ can be proved similarly.
    \end{enumerate}
             \item[(c)] If $J=\tau(K)$ then $g_{K,J}(m)$  strictly increases on $(-1,1)$ by Proposition \ref{Properties of g}. Therefore, \eqref{m=tanh} has a unique solution which is the global maximizer of $\phi_{K,J,h}(m)$. Hence, the claims (c1) and (c2) hold. 
  \end{enumerate}  
We finally check the regularity order of the maximizers. Using Corollary \ref{Regularity of maximizers} and the above arguments, for $i\in\{0,1,2\}$, the maximizer $m_i$ is 2-regular. On the other hand, in the case (c1), we have $\phi_{K,J,h}''(\overline{m}_0)=0$ (by Corollary \ref{Regularity of maximizers}) and $\phi_{K,J,h}^{(4)}(\overline{m}_0)<0$ (by \eqref{fourth der}). Therefore, $\overline{m}_0$ is a 4-regular maximizer of $\phi_{K,J,h}$.
\end{proof}
\subsection{The coexistence surface}\label{sec:surface}
In the case that $J>\tau(K)$ and $\underline{\Psi}(K,J)<h<\overline{\Psi}(K,J)$ as described in Proposition \ref{Classification of critical points}, the function $\phi_{K,J,h}(m)$ has two local maximizers $m_1 (K,J,h)<m_2(K,J,h)$.  To identify the coexistence of two global maximizers of $\phi_{K,J,h}(m)$, we consider the following function
\begin{equation}
\Delta(K,J,h)=\begin{cases}
    \phi_{K,J,h}(m_2)-\phi_{K,J,h}(m_1),& \text{if } J>\tau(K),\; \underline{\Psi}(K,J)<h<\overline{\Psi}(K,J)\\
    \phi_{K,J,h}(\underline{m})-\phi_{K,J,h}(m_1),& \text{if } J>\tau(K),\; h=\underline{\Psi}(K,J)\\
    \phi_{K,J,h}(m_2)-\phi_{K,J,h}(\overline{m}),& \text{if } J>\tau(K),\; h=\overline{\Psi}(K,J)
\end{cases}.
\end{equation}
The next proposition constructs the coexistence surface by comparing the heights of the two local maxima.
    \begin{proposition}\label{Surface S}
    The function $\Delta(K,J,h)$ is continuous.  
        For each $K>0$ and $J>\tau(K)$, the equation $\Delta(K,J,h)=0$ has a unique solution denoted by $S(K,J)$. Moreover, $S(K,J)\in (\underline{\Psi}(K,J),\overline{\Psi}(K,J))$ and $S$ is smooth on $\{(K,J):K>0,\,J>\tau(K)\}$.
    \end{proposition}
    \begin{proof}
    The function \(\Delta\) is continuous in the open three-root region, since the two local maximizers \(m_1\) and \(m_2\) depend smoothly on \((K,J,h)\) by the implicit function theorem. Moreover, as \(h\) approaches the lower (resp. upper) boundary field, \(m_2\) (resp. \(m_1\)) converges to the corresponding turning point \(\underline m\) (resp. \(\overline m\)), while the other branch remains continuous. Thus, the definition of \(\Delta\) at the boundary fields gives a continuous extension.
    
    On the other hand, by Proposition \ref{Classification of critical points}, if $J>\tau(K)$ then $m_1$ (resp. $m_2$) is the unique global maximizer of $\phi_{K,J,h}(m)$ when $h=\underline{\Psi}(K,J)$ (resp.  $h=\overline{\Psi}(K,J)$). This implies that
    \begin{equation}
        \Delta(K,J,h)\begin{cases}
         <0, \text{ when } h=\underline{\Psi}(K,J)\\
            >0, \text{ when }  h=\overline{\Psi}(K,J)
        \end{cases}.
    \end{equation}
Therefore,  the equation $\Delta(K,J,h)=0$ has at least  one solution in  $(\underline{\Psi}(K,J), \overline{\Psi}(K,J))$ since $\Delta$ is continuous in $h$. Next, we prove the uniqueness.   For  $i \in \{1,2\}$,  $m_i$  is the solution of equation $\dfrac{\partial}{\partial m}\phi_{K,J,h}(m)=0$ and $\dfrac{\partial ^2}{\partial m^2} \phi_{K,J,h}(m_i)<0$. Therefore, the  implicit function theorem (see e.g. \cite[Proposition~A.2]{CMO24}) implies that  $m_i$ is a smooth function of $(K,J,h)$ in the three-root region for $i=1,2$. By the chain rule and the stationarity of $m_i$,
     \begin{equation}
     \frac{\partial}{\partial h}\phi_{K,J,h}(m_i(K,J,h))
     =m_i+\phi'_{K,J,h}(m_i)\frac{\partial m_i}{\partial h}=m_i.
     \end{equation}
Hence,
        \col{\eqan{\label{Delta'}
    \dfrac{\partial}{\partial h}\Delta(K,J,h)= \dfrac{\partial}{\partial h} \phi_{K,J,h}(m_2) -  \dfrac{\partial}{\partial h} \phi_{K,J,h}(m_1)=m_2-m_1>0.
    }}
     Hence, the function $\Delta(K,J,h)$ is strictly increasing in $h$, and thus the solution of the equation $\Delta(K,J,h)=0$ is unique. Since the two local maximizing branches are smooth  and $\partial_h\Delta=m_2-m_1>0$, the implicit function theorem also gives the asserted smoothness of $S(K,J)$.
    \end{proof}
The preceding proposition establishes the existence, uniqueness, and smoothness of the function $S(K,J)$ in \eqref{defsur}. Equivalently, $S(K,J)$ is the unique solution of $\Delta(K,J,h)=0$.
    
     \subsection{The critical curve and completion of the proof}\label{sec:critical}
     \begin{proposition}\label{curve gamma}
        Let $\gamma=\{(K(t),J(t),h(t)): t \in (0,1)\}$ be the curve defined in \eqref{defgamma}. Then for $(K,J,h)\in\gamma$, the global maximizer of $\phi_{K,J,h}$ is $4$-regular. 
    \end{proposition}
    \begin{proof}
       For $(K,J,h)\in\mathcal{P}$, if a maximizer of $\phi_{K,J,h}(m)$ is 4-regular, then this maximizer is the solution of the following system of equation
      \begin{align}
		\begin{cases}
			\phi'_{K,J,h}(m)=Km^2+Jm+h-\arctanh(m)=0\\
			\phi''_{K,J,h}(m)=2Km+J-\dfrac{1}{1-m^2}=0\\
			\phi^{(3)}_{K,J,h}(m)=2K-\dfrac{2m}{(1-m^2)^2}=0,
		\end{cases}
	\end{align}
or equivalently 
    $$\begin{cases}
		K=\dfrac{m}{(1-m^2)^2}\\
		J=\dfrac{1-3m^2}{(1-m^2)^2}\\
		h=\mathrm{atanh}(m)-\dfrac{m-2m^3}{(1-m^2)^2}.
	\end{cases}$$
Since we are restricted to the case $K \in (0, \infty)$, in the above system only $m \in (0,1)$ is admissible. Thus a quartically degenerate maximizer is possible only when $(K,J,h)$ belongs to the curve $\gamma$ explicitly defined in \eqref{defgamma}. Conversely, at every point of this curve,
we have $J=\tau(K)$ and $m=\overline m_0(K)$; by the strict concavity away from $m$, $m$ is the unique global maximizer, and it is $4$-regular.
    \end{proof}
   \begin{proof}[Proof of Theorem \ref{Phase diagram}]
       Let $\kS$ and $\gamma$ be the surface and curve of parameters defined in \eqref{defsur}  and \eqref{defgamma} respectively. Denote
       \begin{align*}
			D_0=&\{(K,J,h)\in\mathcal{P}: J\le \tau(K)\}\backslash \gamma\\
			D_1=&\{(K,J,h)\in\mathcal{P}: J> \tau(K), h<S(K,J)\}\\
			D_2=&\{(K,J,h)\in\mathcal{P}: J> \tau(K), h>S(K,J)\}
		\end{align*}
        For any $(K,J,h)\in\mathcal{P}$, using Propositions \ref{Classification of critical points}, \ref{Surface S} and \ref{curve gamma}, we have the following.
        \begin{enumerate}
            \item [(a)] If $(K,J,h)\in\mathcal{P}\backslash \mathcal{S}$, $\phi_{K,J,h}(m)$ has a unique global maximizer $m_*$ identified as following:
            \begin{align}
                m_*=  \begin{cases}
               m_i, &\text{ if } (K,J,h)\in D_i, \; i\in\{0,1,2\}\\
               \overline{m}_0, &\text{ if } (K,J,h)\in\gamma.
            \end{cases}
            \end{align}
           Moreover,  $\overline{m}_0$ is $4$-regular, and otherwise it is $2$-regular. 
           \item[(b)] If $(K,J,h)\in \mathcal{S}$, $\phi_{K,J,h}(m)$ has two global maximizers $m_1$ and $m_2$, which are both 2-regular.
        \end{enumerate}
It remains to prove the closure assertion: 
$$\overline{\mathcal S}^{\,\mathcal P}
=\mathcal S\cup\gamma.$$ 
Fix $(K_0,J_0,h_0)\in\gamma$. By
Proposition~\ref{curve gamma}, $J_0=\tau(K_0)$ and
$h_0=g_{K_0,J_0}(\overline{m}_0(K_0))$. For $J>\tau(K_0)$ let
$a_J<b_J$ be the two solutions of $F_{K_0}(m)=J$. Since
$F_{K_0}$ is strictly convex and has a unique minimizer $\overline{m}_0(K_0)$,
both $a_J$ and $b_J$ converge to $\overline{m}_0(K_0)$ as
$J\downarrow\tau(K_0)$. Consequently,
\[
\underline\Psi(K_0,J)=g_{K_0,J}(b_J)\rightarrow h_0,
\qquad
\overline\Psi(K_0,J)=g_{K_0,J}(a_J)\rightarrow h_0.
\]
Since $\underline\Psi(K_0,J)<S(K_0,J)<\overline\Psi(K_0,J)$,
we get $S(K_0,J)\to h_0$. Thus $\gamma\subset\overline{\mathcal S}^{\,\kP}$.

Conversely, suppose that
$(K_n,J_n,S(K_n,J_n))\in\mathcal S$ converges in $\kP$ to
$(K_0,J_0,h_0)$. Continuity of $\tau$ gives
$J_0\ge\tau(K_0)$. If $J_0>\tau(K_0)$, continuity of $S$
from Proposition~\ref{Surface S} shows that
$h_0=S(K_0,J_0)$, so the limit belongs to $\mathcal S$.
It remains to consider $J_0=\tau(K_0)$. Let $a_n<b_n$ be the
solutions of $F_{K_n}(m)=J_n$. These points stay in a compact
subinterval of $(-1,1)$: the parameters $(K_n,J_n)$ are bounded,
whereas $F_K(m)\to+\infty$ as $m\to\pm1$, uniformly for any bounded number $K$.
Every accumulation point of $a_n$ or $b_n$ solves
$F_{K_0}(m)=\tau(K_0)$, and hence equals the unique minimizer
$\overline{m}_0(K_0)$. Therefore $a_n,b_n\to \overline{m}_0(K_0)$ and by continuity
\[
\underline\Psi(K_n,J_n),\ \overline\Psi(K_n,J_n)
\rightarrow g_{K_0,\tau(K_0)}(\overline{m}_0(K_0)).
\]
As $S(K_n,J_n)$ lies strictly between these two quantities,
$h_0=g_{K_0,\tau(K_0)}(\overline{m}_0(K_0))$. By the parametrization
\eqref{defgamma}, $(K_0,J_0,h_0)\in\gamma$.
This proves the closure assertion.
\end{proof}

\section{Fluctuation of the magnetization and maximum likelihood estimators}\label{sec:applications}
In this section, all parameters are fixed in $\kP$. The geometry established in Theorem~\ref{Phase diagram} allows us to apply the general mean field fluctuation and maximum likelihood theory of \cite{CR}. We state the model specific limits explicitly and verify the nondegeneracy conditions for the estimators, while using Remark~2.2 of \cite{CR} for the standard mean field approximation assumptions.

\subsection{Magnetization fluctuations}\label{sec:mag}
\paragraph{How the fluctuation scales arise.}
The magnetization takes values $m_k=2k/N-1$, and its exact mass is proportional to
\begin{equation}\label{mag-saddlepoint}
    \binom{N}{k}\exp\{N u_{K,J,h}(m_k)\},\qquad
    u_{K,J,h}(m)=\frac K3m^3+\frac J2m^2+hm.
\end{equation}
Stirling's approximation shows that, near any interior global maximizer $m$, the logarithm of this mass is $N\phi_{K,J,h}(m_k)$ up to a slowly varying prefactor. If $m$ is $2r$-regular, its local expansion reads
\begin{equation}\label{mag-local-taylor}
N[\phi(m+xN^{-1/(2r)})-\phi(m)]
=\frac{\phi^{(2r)}(m)}{(2r)!}x^{2r}+o(1),
\end{equation}
for fixed $x$; throughout this discussion $\phi=\phi_{K,J,h}$. Thus the scale $N^{-1/(2r)}$ is obtained by balancing the loss in free energy against the factor $N$ in the Gibbs weight.

For a $2$-regular maximizer, $r=1$ and the leading term is quadratic; the limiting local density is Gaussian with variance $-1/\phi''(m)$. When the unique maximizer is $4$-regular, $r=2$ and the quadratic and cubic coefficients vanish; the negative fourth derivative produces a density proportional to $\exp\{\phi^{(4)}(m)x^4/24\}$ on the scale $N^{-1/4}$. The critical scaling is therefore a consequence of the local shape of the pressure, not of a different normalization of the Hamiltonian.

The presence of several global maximizers changes the \emph{unconditioned} limit. On $\mathcal S$, the two wells have the same free energy height and both survive as $N\to\infty$. Their leading Laplace contributions are proportional to
\begin{equation}\label{mag-well-Laplace}
\frac{1}{\sqrt{1-m_i^2}}
\int_{\mathbb R}\exp\left\{\frac12\phi''(m_i)x^2\right\}\,dx,
\qquad i=1,2.
\end{equation}
Consequently, $m_N$ has a two-point mixture limit, whereas its fluctuations \emph{conditional on one well} remain Gaussian at scale $N^{-1/2}$. This distinction between unconditional phase selection and conditional fluctuations is important. The statements below also give concentration estimates and Wasserstein error bounds, obtained from \cite{CR}.

For probability laws with finite first moments, write $d_W$ for their $1$-Wasserstein distance. Let $\mathcal M=\operatorname*{argmax}_{m\in[-1,1]}\phi_{K,J,h}(m)$. For each $2$-regular maximizer $m$, define
\begin{equation}\label{gaussian-variances}
\sigma_m^2=-\frac{1}{\phi_{K,J,h}''(m)}>0.
\end{equation}
On $\kS$, we have $|\mathcal{M}| = 2$, and let $\mathcal M=\{m_1,m_2\}$ with $m_1<m_2$ and set
\begin{equation}\label{explicit-weights}
w_i=\frac{[(1-m_i^2)(-\phi_{K,J,h}''(m_i))]^{-1/2}}
{\sum_{j=1}^2[(1-m_j^2)(-\phi_{K,J,h}''(m_j))]^{-1/2}},\qquad i=1,2.
\end{equation}

\begin{theorem}[Magnetization limit laws]\label{limit:magn}
Let $m_N$ be the magnetization under $\mu_{N,K,J,h}$. The following hold, with constants depending on fixed parameters and, where appropriate, on $\delta$.
\begin{enumerate}
\item[(a)]  If $(K,J,h)\in\kP\setminus(\kS\cup\gamma)$ and $m_*$ is the maximizer, then  for every sufficiently small $\delta>0$,
\begin{equation}\label{mag-regular-conc}
\mu_N(|m_N-m_*|>\delta)\le e^{-cN},\qquad
d_W\bigl(\mathcal L(\sqrt N(m_N-m_*)),\,N(0,\sigma_{m_*}^2)\bigr)\le CN^{-1/2}.
\end{equation}
\item[(b)]  If $(K,J,h)\in\kS$, then
\begin{equation}\label{mag-two-wells}
m_N\Rightarrow w_1\delta_{m_1}+w_2\delta_{m_2},\qquad
\mu_N\bigl(\operatorname{dist}(m_N,\{m_1,m_2\})>\delta\bigr)\le e^{-cN}.
\end{equation}
For $i=1,2$ and all sufficiently small $\delta>0$,
\begin{equation}\label{mag-coexist-CLT}
d_W\bigl(\mathcal L(\sqrt N(m_N-m_i)\mid |m_N-m_i|<\delta),\,
N(0,\sigma_{m_i}^2)\bigr)\le CN^{-1/2}.
\end{equation}
\item[(c)]  If $(K,J,h)\in\gamma$ and $m_*$ is the unique maximizer, then
\begin{equation}\label{mag-quartic}
\mu_N(|m_N-m_*|>\delta)\le e^{-cN}, \qquad d_W\bigl(\mathcal L(N^{1/4}(m_N-m_*)),\,\mathcal L(Y_*)\bigr)\le CN^{-1/4},
\end{equation}
where $Y_*$ has density proportional to $\exp(c_*x^4)$ and
\begin{equation}\label{quartic-coeff}
c_*:=\frac{\phi_{K,J,h}^{(4)}(m_*)}{4!}
=-\frac{1+3m_*^2}{12(1-m_*^2)^3}<0.
\end{equation}
\end{enumerate}
\end{theorem}

\begin{proof}
Set $a=(1+m)/2$ and
\[
F(a)=\frac K3(2a-1)^3+\frac J2(2a-1)^2+h(2a-1).
\]
The distribution of $X_N=N(1+m_N)/2$ has the form
\[
\mathbb P(X_N=k)=\frac{1}{Z_N}\binom Nk e^{NF(k/N)},\qquad 0\le k\le N.
\]
Remark~2.2 of \cite{CR} applies to this smooth mean-field energy: the
binomial term provides the finite-volume estimates (A2)--(A4), with the
prefactors at the maximizers as given there. Assumption (A1), namely the
location and regularity orders of the global maximizers of
$F(a)-a\log a-(1-a)\log(1-a)$, follows from
Theorem~\ref{Phase diagram} under the change of variables $m=2a-1$.
Thus Theorems~2.3--2.5 of \cite{CR} apply directly (with zero perturbation).

At a $2$-regular maximizer $m_i$, the limiting local density on the
$\sqrt N(m_N-m_i)$ scale is proportional to
$\exp\{\phi''(m_i)x^2/2\}$, giving variance
$-1/\phi''(m_i)$. On $\mathcal S$, the weights from
\cite[Remark~2.2 and Theorem~2.3]{CR} are proportional to
\[
\frac{1}{\sqrt{a_i(1-a_i)}}
\int_{\mathbb R}\exp\{\tfrac12\phi''(m_i)x^2\}\,dx,
\qquad a_i=\frac{1+m_i}{2},
\]
which reduces to \eqref{explicit-weights}. Finally, on $\gamma$,
Theorem~\ref{Phase diagram} gives a unique $4$-regular maximizer;
the quartic coefficient is $\phi^{(4)}(m_*)/24$.
The concentration and Wasserstein estimates in the statement follow
from Theorems~2.3--2.5 of \cite{CR}.
\end{proof}

\subsection{One-parameter maximum likelihood estimators}\label{sec:mle}
\paragraph{The statistical problem and definition of the estimator.}
For $N\ge2$, we observe \emph{one full configuration} $\sigma\in\{-1,+1\}^N$ sampled from $\mu_{N,K,J,h}$. We estimate one parameter $\theta\in\{K,J,h\}$ at a time, holding the other two fixed and known.  Each one-parameter submodel is an exponential family with sufficient statistic $f_\theta(m_N)$,  since  the likelihood of a configuration depends on the data only through $m_N$. More precisely, we write
\begin{equation}\label{profile-functions}
\mu_{N,\theta}(\sigma)=\frac{1}{Z_N(\theta)}
\exp\{N[\theta f_\theta(m_N)+g_\theta(m_N)]\},
\end{equation}
where
\begin{align}\label{profile-sufficient}
f_K(m)&=m^3/3,& g_K(m)&=Jm^2/2+hm,\\
f_J(m)&=m^2/2,& g_J(m)&=Km^3/3+hm,\\
f_h(m)&=m,    & g_h(m)&=Km^3/3+Jm^2/2.
\end{align}
The (per-spin) log likelihood is
\begin{equation}\label{log-likelihood}
\ell_N(\theta;\sigma)=\frac1N\log\mu_{N,\theta}(\sigma)
=\theta f_\theta(m_N)+g_\theta(m_N)-p_N(\theta),
\qquad p_N(\theta)=\frac1N\log Z_N(\theta).
\end{equation}
Maximizing $\ell_N$ is equivalent to matching the observed sufficient statistic with its model expectation. More precisely, the score and its derivative satisfy
\begin{equation}\label{score-correct}
\partial_\theta\bigl(N^{-1}\log\mu_{N,\theta}(\sigma)\bigr)
=f_\theta(m_N)-d_N(\theta),\qquad d_N(\theta)=\mathbb E_\theta[f_\theta(m_N)],
\end{equation}
\begin{equation}\label{MLE-variance-identity}
d_N'(\theta)=N\operatorname{Var}_\theta\bigl(f_\theta(m_N)\bigr)>0.
\end{equation}
For $N\ge2$, the statistic $f_\theta(m_N)$ is nonconstant under the Gibbs measure, so the variance in \eqref{MLE-variance-identity} is strictly positive. Hence the score equation has at most one solution, and the log likelihood is strictly concave. Whenever the observed statistic lies strictly between its minimal and maximal possible lattice values, the score equation has a unique finite solution; the \emph{unconstrained} one-parameter MLE is
\begin{equation}\label{def:mle}
\widehat\theta_N=d_N^{-1}(f_\theta(m_N)).
\end{equation}
When the observed statistic attains its minimum or maximum possible value, a finite MLE does not exist. In the regimes considered below, these events have exponentially small probability, and we may assign an arbitrary finite value to the estimator on them without affecting the limiting distributions. 
\paragraph{Why the estimation rates depend on the phase.}
The score representation also indicates how the magnetization limit enters the MLE. Suppose first that $m$ is the unique $2r$-regular maximizer and that $f_\theta'(m)\ne0$. Under a local change $\theta\mapsto\theta+tN^{-\alpha}$, the additional exponent in the Gibbs weight is $tN^{1-\alpha} f_\theta(m_N)$. The constant part $f_\theta(m)$ cancels after normalization. On the magnetization scale $m_N=m+N^{-1/(2r)}x$, the remaining exponent becomes
\begin{equation}\label{mle-local-tilt-mechanism}
 t N^{1-\alpha-1/(2r)}f_\theta'(m)x+o(1).
\end{equation}
The perturbation thus has an order one effect precisely when $\alpha=1-1/(2r)$. For a regular well this gives the classical $N^{-1/2}$ parameter scale and a Gaussian MLE limit; for a quartic well this gives $N^{-3/4}$ and a non-Gaussian limit determined by the corresponding tilted quartic density. 

On the coexistence surface, the mechanism is different. Let $m_-$ and $m_+$ denote the maximizing wells ordered by $f_\theta(m_-)<f_\theta(m_+)$. A local perturbation $t/\sqrt N$ shifts their relative log weights by
$ t\sqrt N [f_\theta(m_+)-f_\theta(m_-)]$. Hence, the positive perturbation selects $m_+$ and the negative one selects $m_-$. Conditioned on the well selected, the fluctuations of the statistic are Gaussian, but the original sample may have come from either well. Inverting the score therefore produces two one-sided normal components and an atom at zero. The weights depend on the two well probabilities; the atom has mass $1/2$ in the nondegenerate two-well situation here. This also explains why it is necessary to check both $f_\theta'(m_i)\ne0$ and $f_\theta(m_1)\ne f_\theta(m_2)$ before applying \cite[Theorem~4.1]{CR}.

For the cubic sufficient statistics, $f_K'(m)=m^2$, $f_J'(m)=m$, and $f_h'(m)=1$. Thus nondegeneracy fails for estimation of $K$ or $J$ exactly when $m=0$ is a global maximizer, while it never fails for estimation of $h$. The relevant exceptional set can be described explicitly from the phase diagram as follows.  Set
\begin{equation}\label{Rzero}
\mathcal R_0=\{(K,J,0)\in\kP:J\le\tau(K)\}
\cup\{(K,J,0)\in\kP:\tau(K)<J<1,\ S(K,J)\ge0\},
\end{equation}
and define
\begin{equation}\label{Rbeta}
\mathcal R_K=\mathcal R_J=\mathcal R_0,\qquad \mathcal R_h=\varnothing.
\end{equation}
The following result shows that for $\theta\in\{K,J,h\}$, the quantity $f_\theta'(m)$ may vanish at a global maximizer only on $\mathcal R_\theta$.

\begin{lemma}[Nondegeneracy and separation of wells]\label{region of g.max. 0}
The point $m=0$ is a global maximizer of $\phi_{K,J,h}$ exactly when $(K,J,h)\in\mathcal R_0$. Moreover, on $\kS$, the two global maximizers satisfy $f_\theta(m_1)\ne f_\theta(m_2)$ for every $\theta\in\{K,J,h\}$.
\end{lemma}
\begin{proof}
A global maximizer at zero must satisfy $h=\phi'(0)=0$. If $J\le\tau(K)$, the stationary point $m=0$ is unique and hence globally maximizing. If $J>1$, then $\phi''(0)=J-1>0$ and if $J=1$, then $\phi''(0)=0$ and $\phi^{(3)}(0)=2K>0$. Thus zero is not a maximizer in either of those cases. For $\tau(K)<J<1$, the left turning point $\overline m$ of $g_{K,J}$ is strictly positive since $F_K(0)=1>J$ and $F_K$ decreases toward its unique minimizer at a positive point. Hence zero lies on the strictly concave left branch. As the difference between the two local maximum heights increases strictly with $h$, the left branch is globally maximizing at $h=0$ precisely when $0\le S(K,J)$. This proves the first claim.

For $\theta=K$ or $h$, $f_\theta$ is strictly increasing, so $f_\theta(m_1)\ne f_\theta(m_2)$. For $\theta=J$, equality would require $m_1=-a$ and $m_2=a$ for some $a>0$. The two stationarity equations force $h=-Ka^2$ and $J=\operatorname{arctanh}(a)/a$; but then
\begin{equation}
\phi(a)-\phi(-a)=\frac{2K}{3}a^3+2ha=-\frac{4K}{3}a^3\ne0,
\end{equation}
contradicting coexistence.
\end{proof}

\begin{theorem}[One-parameter MLE limits]\label{limit:mle}
Fix $(K,J,h)\in\kP$ and let $\widehat\theta_N$ be defined by \eqref{def:mle}, with the other two parameters known. Then for every $\theta\in\{K,J,h\}$ the following hold.
\begin{enumerate}
\item[(a)]  If $(K,J,h)\in\kP\setminus(\kS\cup\gamma\cup\mathcal R_\theta)$, then
\begin{equation}\label{Ua}
\sqrt N(\widehat\theta_N-\theta)\Rightarrow
\mathcal N\left(0,\frac{-\phi_{K,J,h}''(m_*)}{[f_\theta'(m_*)]^2}\right).
\end{equation}
\item[(b)] If $(K,J,h)\in\kS\setminus\mathcal R_\theta$, let $m_-$ and $m_+$ be the two global maximizers in increasing order of $f_\theta(m)$; write $w_-,w_+$ for their weights in \eqref{explicit-weights} and $v_\pm=-\phi''_{K,J,h}(m_\pm)/[f_\theta'(m_\pm)]^2$. Then
\begin{equation}\label{Ub}
\sqrt N(\widehat\theta_N-\theta)\Rightarrow U_b,
\qquad\mathcal L(U_b)=\frac12\delta_0+
\frac{w_-}{2}\mathcal L(-|Z_-|)+\frac{w_+}{2}\mathcal L(|Z_+|),
\end{equation}
where $Z_\pm\sim N(0,v_\pm).$
\item[(c)]  If $(K,J,h)\in\gamma$, then
\begin{equation}\label{Uc}
N^{3/4}(\widehat\theta_N-\theta)\Rightarrow U_c,\qquad
\mathbb P(U_c\le t)=\mathbb P\bigl\{Y_*(0)\le\mathbb E[Y_*(t)]\bigr\},\quad t\in\mathbb R,
\end{equation}
where $Y_*(t)$ has a density proportional to
$\exp\{c_*x^4+t f_\theta'(m_*)x\}$ with $c_*$ from \eqref{quartic-coeff}.
\end{enumerate}
\end{theorem}

\begin{proof}
We verify the conditions of the general MLE theorem \cite[Theorem~4.1]{CR}. The representation \eqref{profile-functions} is linear in the estimated parameter, and $f_\theta,g_\theta$ are smooth polynomials. The model has finitely many interior maximizers with regularity orders $2$ or $4$ by Theorem~\ref{Phase diagram}; the finite-volume conditions for this mean-field model and for the polynomial perturbations $t f_\theta$ are covered by \cite[Remark~2.2]{CR}. On $\kP\setminus\mathcal R_\theta$, Lemma~\ref{region of g.max. 0} ensures $f_\theta'(m)\ne0$ at each maximizing well. On coexistence, the maximizing wells have equal regularity order and distinct $f_\theta$-values, so the extremal well and nondegeneracy conditions in \cite[Theorem~4.1]{CR} hold. On $\gamma$ the maximizer satisfies $m_*\in(0,1)$, hence $f_\theta'(m_*)>0$ for all three parameters.

Therefore, all the conditions of \cite[Theorem~4.1]{CR} are satisfied. Applying this theorem, we obtain the scaling limits stated in Theorem~\ref{limit:mle}: the MLEs have scale $N^{-1/2}$ in the regular and coexistence regimes, and $N^{-3/4}$ on the critical curve $\gamma$.\\

For completeness, we briefly explain how the
limiting distributions take their explicit forms
in the present model. If $m$ is a unique $2k$-regular maximizer, perturb the unknown parameter by $\theta_{N,t}=\theta+tN^{-(1-1/(2k))}$. The distribution of the rescaled magnetization under the perturbed law converges to a random variable $Y_m(t)$ with density proportional to
\begin{equation}\label{tilted-law}
\exp\left\{\frac{\phi^{(2k)}_{K,J,h}(m)}{(2k)!}x^{2k}+t f_\theta'(m)x\right\}.
\end{equation}
The moment estimates of \cite{CR} give the corresponding first-moment convergence. Since $d_N$ is strictly increasing,
\begin{equation}\label{MLE-generic-CDF}
\mathbb P\{N^{1-1/(2k)}(\widehat\theta_N-\theta)\le t\}
\longrightarrow\mathbb P\{f_\theta'(m)Y_m(0)\le f_\theta'(m)\mathbb E[Y_m(t)]\}.
\end{equation}
For $k=1$, Gaussian exponential tilting yields \eqref{Ua}. For $k=2$, the maximizer is on $\gamma$ and $f_\theta'(m)>0$, giving \eqref{Uc}. The scaling $N^{3/4}$ is the parameter-estimation scale; the magnetization itself has scale $N^{-1/4}$.

On $\kS$, write $\Phi$ for the standard normal distribution function. Because $f_\theta(m_+)>f_\theta(m_-)$, the perturbed law with parameter $\theta+t/\sqrt N$ selects the $m_-$ well for fixed $t<0$ and the $m_+$ well for fixed $t>0$. The score identity shows, for each $t\ne0$, that the limiting distribution function is
\begin{equation}\label{MLE-coexist-CDF}
\lim_{N\to\infty}\mathbb P_\theta\{\sqrt N(\widehat\theta_N-\theta)\le t\}
=\begin{cases}
 w_-\Phi(t/\sqrt{v_-}),&t<0,\\
 w_-+w_+\Phi(t/\sqrt{v_+}),&t>0.
\end{cases}
\end{equation}
Indeed, on the well selected by the perturbation the score threshold has a Gaussian displacement, while observations in the opposite well fall entirely below or above that threshold. The limiting left and right values at zero are $w_-/2$ and $w_-+w_+/2$, respectively; their difference is $1/2$. Formula \eqref{MLE-coexist-CDF} is the distribution function of the half-normal/atomic mixture in \eqref{Ub}.
\end{proof}

\noindent \textbf{Acknowledgement}. The authors thank the Vietnam Institute for Advanced Study in Mathematics (VIASM) for its hospitality and the fruitful research environment, where part of this work was carried out.

\end{document}